\documentclass[12pt, reqno]{amsart}

\usepackage{style}

\begin{document}
\mbox{}
\title{A Projective Twist on the
Hasse Norm Theorem}

\author{Alan Bu}
\address{Harvard College\\Cambridge, MA, 02138}
\email{abu@college.harvard.edu}

\author{Thomas R\"ud}
\address{Massachusetts Institute of Technology\\Cambridge, MA, 02139}
\email{rud@mit.edu}

\author{Kevin Xiao}
\address{Archmere Academy\\ Claymont, DE, 19703}
\email{kjxiao9@gmail.com}

\begin{abstract}
A finite extension of global fields $L/K$ satisfies the Hasse norm principle if any nonzero element of $K$ has the property that it is a norm locally if and only if it is a norm globally. In 1931, Hasse proved that any cyclic extension satisfies the Hasse norm principle, providing a novel approach to extending the local-global principle to equations with degree greater than $2$. In this paper, we introduce the projective Hasse norm principle, generalizing the Hasse norm principle to multiple fields and asking whether a line through the origin that contains a simultaneous norm locally in every field must also contain a norm globally in every field. We show that the projective Hasse norm principle is independent of the conjunction of Hasse norm principles in all of the constituent fields in the general case, and that the projective principle holds whenever the fields are all Galois and jointly linearly disjoint, without assuming their individual Hasse norm principles. We also give a criterion in terms of their Galois groups and prove that the implication holds whenever the Galois group of their compositum has fourth-power-free order. A counterexample of Demarche and Wei shows that this bound is sharp. We also prove an analogue of the Hasse norm theorem for the projective Hasse norm principle, namely that the projective Hasse norm principle holds for every finite family of cyclic extensions.
\end{abstract}

\maketitle

\tableofcontents
%%%%%%%%%%%%%%%%%%%%
%%%%%%%%%%%%%%%%%%%%

\section{Introduction}
\label{sec:intro}
Nearly a century ago, Helmut Hasse helped develop the \emph{local-global principle}: the idea that, if a diophantine equation has a solution over every completion of a global field, i.e. it has a solution \emph{locally} everywhere, then it has a solution \emph{globally} over that field. Since then, the principle has been carefully studied in a variety of spaces, and to great depths. The principle does not hold in general: in {\cite{selmer}}, Selmer's famous counterexample is the equation $3x^3+4y^3+5z^3 = 0$. It is not hard to show this equation has a nontrivial solution over every completion of $\mathbb{Q}$, and yet it has no nontrivial solution over the rationals. However, in other instances, the principle holds: for instance, the Hasse--Minkowski theorem proves that a quadratic form over the rationals has a nontrivial zero if and only if it has one over every completion. Quadratic forms are multivariable polynomial expressions whose monomials all have degree two. Selmer's example shows that one cannot extend the Hasse--Minkowski theorem to arbitrary polynomial expressions of higher degrees, but there are many approaches to extend the local-global principle into cases involving higher degrees.

One of the most prominent of these approaches is through the use of multiplicative norms. In any finite field extension, we can assign each element a multiplicative norm. For instance, in the Gaussian rationals over $\mathbb{Q}$, the norm of an element is the same as the square of its complex magnitude. In 1931, Hasse published his acclaimed Hasse norm theorem {\cite{hasse}}: given a cyclic extension of the rationals $K/\mathbb{Q}$, he proved that the norm of the number field $K$ satisfies the local-global principle: given any nonzero rational number $q$, if the equation $N_{K/\mathbb{Q}}(x) = q$ has a solution everywhere locally then it has a solution globally.

In our research, we put a projective twist on the problem. Instead of looking at the local-global principle at specific points, we look at lines through the origin instead. Given a space of several field extensions, if a line through the origin contains a nonzero local solution at every place, when must it also contain a global solution?

To address this question, in \S\ref{sec:statement} we introduce the Hasse norm principle ($\HNP$), the Multinorm Problem, as well as the Projective Hasse norm principle ($\PHNP$). We fix standardized notations in \S \ref{sec:notations} and introduce the methods used in \S \ref{sec:methodology}. 

In \S\ref{sec:elementary} we show a few preliminary results around simple cases of the projective Hasse norm principle and its connection to the Hasse norm principle that do not require cohomology. In \S\ref{sec:description} we derive a closed form for the character lattices of tori closely related to the projective Hasse norm principle, which we use extensively to analyze our research problem in later sections.

In \S\ref{sec:counterexamplephnptohnp}, we show that $\PHNP$ does not imply $\HNP$ by giving an explicit counterexample with a biquadratic extension. On the other hand, it is often true that when $\HNP$ is true in all of the constituent fields, then $\PHNP$ holds as well. To concretely show this, we derive a simple sufficient condition for $\HNP \Longrightarrow \PHNP$ in \S \ref{sec:condition} depending on the Galois group of the compositum and the subgroups corresponding to the constituent fields. We also derive an explicit construction dual to the Tate--Shafarevich group that encodes the $\PHNP$ condition in terms of the decomposition groups of the composite field extension, which allows us to reframe the obstruction in terms of finite-group cohomology and decomposition groups. Using these tools, we show that when all of the constituent fields are Galois and jointly linearly disjoint, then $\PHNP$ holds without assuming their individual Hasse norm principles.

Using these techniques, in \S \ref{sec:counterexamplehnptophnp} we construct an explicit counterexample to $\HNP \Longrightarrow \PHNP$ using non-Galois constituent fields. For two fields, $\PHNP$ is exactly the multinorm principle; we also explain how the Galois counterexample of \cite[Proposition~12 and Example~13]{demarche-wei} applies to our question. In \S \ref{sec:galois} we further study the implication $\HNP \Longrightarrow \PHNP$ in the Galois case, and show that regardless of the choice of the constituent fields, when the Galois group of the compositum is abelian, dihedral, or has fourth-power-free order, then it follows that $\HNP \longrightarrow \PHNP$. The fourth-power-free result includes groups of order $p^3$. The general Galois implication is false: the example of Demarche and Wei has compositum of degree $16$, which is the smallest possible degree for such a counterexample.

\section{Statement of the Problem and Related Questions}
\label{sec:statement}

Let $K$ be a global field, i.e., a finite extension of $\mathbb{Q}$ or the function field of a smooth curve over a finite field. Denote by $\overline{K}$ its separable closure. For each place $v$ of $K$, let $K_v$ denote the corresponding completion of $K$ and let $\mathbb{A}_K$ denote the ring of ad\`eles of $K$. 

 Let $L/K$ be a finite separable extension. For each place $v$ of $K$ and place $w$ of $L$ such that $w|v$, the norm map $N_{L/K}: L^\times\rightarrow K^\times$ induces the norm map $N_{L_w/K_v} : L_w^\times \rightarrow K_v^\times$ and, by taking the product over $w\mid v$ at each place $v$, a map $\mathbb{A}_L^\times \rightarrow \mathbb{A}_K^\times$ which we will also denote by $N_{L/K}$.

\subsection{The Hasse Norm Principle}

The inclusion $L^\times\subset \mathbb{A}_L^\times$ yields \begin{equation}\label{eq:norminclusion}N_{L/K}(L^\times)\subseteq K^\times\cap N_{L/K}(\mathbb{A}_L^\times).\end{equation} If $L$ is a field, we say the \emph{Hasse norm principle} holds if the inclusion is an equality. We will let $\mathrm{HNP}_K(L)$ be the Boolean equal to True if and only if the Hasse norm principle holds.

\begin{theorem}[Hasse norm theorem, 1931]
$\mathrm{HNP}_K(L)$ holds whenever $L/K$ is cyclic. 
\end{theorem}

The modern study of the Hasse norm principle uses techniques in class field theory to determine when \begin{equation}\label{eq:shanormpple}\Sh(L/K):=(K^\times\cap N_{L/K}(\mathbb{A}_L^\times))/N_{L/K}(L^\times) \cong \Sh^1(K,\mathbf{T}^1_{L/K})\end{equation}
is trivial, where $\mathbf{T}^1_{L/K} = \mathrm{R}_{L/K}^{(1)}\mathbb{G}_m=\mathrm{Ker}(N_{L/K}: \mathrm{R}_{L/K}\mathbb{G}_m\rightarrow \mathbb{G}_m )$ is the norm-one torus corresponding to $L/K$. 
For any $\mathrm{Gal}(\overline{K}/K)$-module $M$, we use the notation
\[\Sh^i(K, M) = \mathrm{Ker}\left( H^i(K, M)\rightarrow \prod_{v}H^i(K_v, M)\right),\]
where $v$ runs over places of $K$. This group is called the $i$th \emph{Tate--Shafarevich group}. For ease of notation, we write $\Sh^i(K,\mathbf{Q})$ for $\Sh^i(K,\mathbf{Q}(\overline K))$ for any torus $\mathbf{Q}$ defined over $K$. If $L/K$ is finite Galois and $M$ is a $\operatorname{Gal}(L/K)$-module, we write
\[\Sh^i(L/K,M)=\ker\left(H^i(\operatorname{Gal}(L/K),M)
 \longrightarrow\prod_v H^i(D_v,M)\right),\]
where $D_v$ is the decomposition group of a chosen place of $L$ above $v$.

The rational points of $\mathbf{T}^1$ are 
\[\mathbf{T}^1_{L/K}(K) = \{x\in L^\times\ :\ N_{L/K}(x) = 1 \}.\]
 The relation (\ref{eq:shanormpple}) is made apparent when looking at the Galois cohomology associated to the sequence 
 \begin{equation}
     1\rightarrow\mathbf{T}^1_{L/K}(\overline{K})\rightarrow \mathrm{R}_{L/K}\mathbb{G}_m(\overline{K})\overset{N_{L/K}}{\rightarrow}\mathbb{G}_m(\overline{K})\rightarrow 1.
 \end{equation}

Moreover, one can compute the obstruction using global duality for tori. If $L^{\sharp}$ is the Galois closure of $L/K$, then
\[\Sh^1(K,\mathbf{T}^1_{L/K})^{\lor}
 \cong\Sh^2(L^{\sharp}/K,\mathbf{X}^{\star}(\mathbf{T}^1_{L/K})),\]
where $A^{\lor}=\operatorname{Hom}(A,\mathbb Q/\mathbb Z)$ and $\mathbf{X}^{\star}$ is the functor mapping a torus to its character lattice. In particular, these two groups vanish simultaneously; see \cite[Theorem~2.3]{macedo-newton} and \cite[\S0]{demarche-wei}.

\subsection{The Multinorm Principle}

More recently, many have studied the \emph{multinorm principle} which treats eq. (\ref{eq:norminclusion}) in the case when $L$ is not a field but an \'etale algebra $L\cong L_1\oplus \cdots\oplus L_k$ where each $L_i/K$ is a field extension. We have $N_{L/K} = \prod_{i=1}^k N_{L_i/K}$ and the multinorm principle holds whenever 
\[\prod_{i=1}^kN_{L_i/K}(L_i^\times) =  K^\times\cap \prod_{i=1}^k N_{L_i/K}(\mathbb{A}_{L_i}^\times).\]

Similarly to the Hasse Norm theorem, one studies this problem by trying to compute the finite group 
\[\Sh(L/K):= \left(K^\times\cap \prod_{i=1}^k N_{L_i/K}(\mathbb{A}_{L_i}^\times)\right)/\left(\prod_{i=1}^kN_{L_i/K}(L_i^\times)\right) \cong \Sh^1(K, \mathbf{S}),\]
where $\mathbf{S} = \mathrm{Ker}(\prod_{i}N_{L_i/K}: \prod_{i} \mathrm{R}_{L_i/K}\mathbb{G}_m\rightarrow \mathbb{G}_m)$. We have 
\[ \mathbf{S}(K) = \left\lbrace (\ell_1,\cdots,\ell_k)\in L_1^\times\times \cdots\times L_k^\times\ :\ N_{L_1/K}(\ell_1)\cdots N_{L_k/K}(\ell_k) = 1\right\rbrace\supset \mathbf{T}^1_{\cap_i L_i/K}(K),\]
where the rightmost group is embedded diagonally, after choosing embeddings of the $L_i$ in the separable closure of $K$.
Let us write a few important results about the multinorm principle.

\begin{theorem}[\cite{pollio-rapinchuk}]
    Let $L_1,L_2$ be finite separable extensions of $K$ with Galois closure $E_1,E_2$ respectively. If $E_1\cap E_2 = K$ then $L = L_1\oplus L_2$ satisfies the multinorm principle.
\end{theorem}
This is proven by studying the surjectivity of the map $\Sh(L_1L_2/K)\rightarrow \Sh(L_1/K)\times\Sh(L_2/K)$. 
\begin{theorem}[\cite{pollio}]
If $L_1,L_2$ are finite abelian field extensions of $K$ then 
\[\Sh(L_1\oplus L_2/K)\cong \Sh(L_1\cap L_2/K).\]
In particular, the multinorm principle corresponding to $L_1\oplus L_2$ holds if and only if   $L_1\cap L_2$ satisfies the Hasse Norm Principle.
\end{theorem}
Pollio-Rapinchuk (\cite{pollio-rapinchuk}) conjectured, for Galois $L_1/K$ and $L_2/K$, that the multinorm principle for $ L = L_1\oplus L_2$ holds whenever every subfield of $L_1\cap L_2$ satisfies the Hasse Norm Principle but a counterexample appeared in \cite[Proposition 12]{demarche-wei}.

\subsection{The Projective Hasse Norm Principle} In this paper, we offer a modification by introducing an intermediate algebra. Let $K\subset E\subset L$ be finite \'etale $K$-algebras, with $L$ finite \'etale over $E$, and let $n=\dim_K E\geq2$. As a $K$-vector space, we have $E\cong K^n$ and therefore a projection map
\[\pi:E^\times\longrightarrow E^\times/K^\times
 \subseteq\mathbb P(E)(K)\cong\mathbb P^{n-1}(K).\]
The inclusion is an equality when $E$ is a field; for $E=K^n$, its image consists of the projective points with every coordinate nonzero. We also write
\[\pi:\mathbb A_E^\times\longrightarrow
 \mathbb A_E^\times/\mathbb A_K^\times.\]

Note that at this stage we do not assume either $L$ or $E$ to be fields. We have $\mathbb A_E=E\otimes_K\mathbb A_K$, and the group of id\`eles of $E$ is the product of the groups of id\`eles of its field constituents. The kernel of the map $E^\times\to\mathbb A_E^\times/\mathbb A_K^\times$ is $E^\times\cap\mathbb A_K^\times=K^\times$, hence $\pi(E^\times)$ embeds into $\pi(\mathbb A_E^\times)$.

We say the \emph{Projective Hasse Norm Principle} holds if 
\[\pi(N_{L/E}(L^\times)) = \pi \left(E^\times\right)\cap \pi\left(N_{L/E}(\mathbb{A}_L^\times)\right).\]

Concretely, the Projective Hasse norm principle corresponding to $(L,E,K)$ holds if the following is true: given $e\in E^\times$, if for every place $u$ of $K$ there is one $\lambda_u\in K_u^\times$ such that
\[\lambda_u e\in
 N_{L\otimes_K K_u\,/\,E\otimes_K K_u}
       ((L\otimes_K K_u)^\times),\]
then there is $\lambda\in K^\times$ such that $\lambda e\in\operatorname{Im}(N_{L/E})$. The same scalar $\lambda_u$ is used in all factors above $u$. At almost all places, $e$ is a unit and the extensions are unramified, so local solutions can be chosen with $\lambda_u=1$ and with unit norm preimages; this also gives the id\`elic formulation above.

Let us rephrase this principle in terms of a \emph{projective norm}.

\begin{definition}[Projective Norm]
    Let $L\supset E$ be finite \'etale $K$-algebras as above. We define the \emph{projective norm} $\mathrm{PN}_{L/E/K}: K^\times \times L^\times\rightarrow E^\times$ by 
    \[(k, \ell)\mapsto k^{-1}N_{L/E}(\ell).\]

    We say the Projective Hasse norm principle relative to $L/E/K$ holds if 
    \[\mathrm{PN}_{L/E/K}(K^\times\times L^\times) = E^\times\cap \mathrm{PN}_{L/E/K}(\mathbb{A}_K^\times\times \mathbb{A}_L^\times).\]
\end{definition}

We are interested in the case $E = K^{\oplus n}$ and $L = \prod_{i=1}^n L_i$ where each $L_i/K$ is a finite separable field extension. We write $\mathrm{PHNP}_K(L_1,\cdots, L_n)$ for the boolean corresponding to the truth of the corresponding Projective Hasse Norm Principle.

Concretely speaking,  $\mathrm{PHNP}_K(L_1,\cdots, L_n)$ is true when the following holds: 
Given a line through a point $(k_1,\ldots,k_n)\in(K^\times)^n$, the line contains a point in the image of $N_{L_1/K}\times\cdots\times N_{L_n/K}$ on $\prod_i L_i^\times$ if and only if, at every place $v$, its scalar extension to $K_v$ contains such a point in the local norm image.

In a similar fashion to the previous problems, we want to determine when 
\[\Sh(L/E/K):= (E^\times\cap \mathrm{PN}_{L/E/K}(\mathbb{A}_K^\times\times \mathbb{A}_L^\times))/\mathrm{PN}_{L/E/K}(K^\times\times L^\times) \cong \Sh^1(K, \mathbf{T})\]
is trivial, where $\mathbf{T} := \mathrm{Ker}\left(\mathbb{G}_m\times \mathrm{R}_{L/K}\mathbb{G}_m\overset{\mathrm{PN}_{L/E/K}}{\longrightarrow} \mathrm{R}_{E/K}\mathbb{G}_m\right)$. The rational points of $\mathbf{T}$ are 
\begin{align*}
    \mathbf{T}(K) &\cong \{(x, y)\in K^\times\times L^\times \ :\ N_{L/E}(y)=x\}\\
    &\cong \left\lbrace y\in L^\times\ :\ N_{L/E}(y)\in K^\times\right\rbrace.
\end{align*}
In our specific setting, we can write 
\begin{align*}
    \mathbf{T}(K) &\cong \{(y_1,\cdots, y_n)\in (L_1\times \cdots\times L_n)^\times :N_{L_i/K}(y_i)=N_{L_j/K}(y_j)\ \forall 1\leq i,j\leq n\}\\
\end{align*}
    We will study this using the cohomology of $\mathbf{T}$  with the cohomology of the subtorus $\mathrm{R}_{E/K}\mathbf{T}^1_{L/E}$, which in our case is isomorphic to $\prod_{i=1}^n\mathbf{T}^1_{L_i/K}$.\\

Firstly, let us note that we can always assume $n>1$.

\begin{prop}
If $E = K$ (i.e. $n = 1$), then the Projective Hasse Norm Principle holds trivially.
\end{prop}
\begin{proof}
    If $E = K$ then we have $\mathbf{T} \cong \mathrm{R}_{L/K}\mathbb{G}_m$ and hence \[\Sh^1(K, \mathbf{T}) \subset H^1(K, \mathbf{T}(\overline{K}))=1.\]
\end{proof}

\begin{remark}
\label{identicalpair}
If $n = 2$, $\PHNP_K(L_1,L_2)$ can be reformulated as the following statement: for any fixed $y\in K^\times$, the equation $y=N_{L_1/K}(x_1)/N_{L_2/K}(x_2)$, with $x_i\in L_i^\times$, has a global solution if and only if it has a solution over every completion. Since $N_{L_2/K}(x_2)^{-1}=N_{L_2/K}(x_2^{-1})$, this is exactly the multinorm principle for $L_1\oplus L_2$. Moreover, if $L_1\subseteq L_2$, transitivity of norms gives
\[\PHNP_K(L_1,L_2)=\HNP_K(L_1).\]
Indeed, globally and at every place the norm image from $L_2$ is contained in that from $L_1$, so the product of the two norm images is the first one. In particular we get the following result.
\end{remark}

\begin{prop}
    If $n=2$ and  $L_1 = L_2$, we have
    \[\PHNP_K(L_1,L_2) = \HNP_K(L_1).\]
\end{prop}
\begin{proof}
    Following Remark \ref{identicalpair}, the equation reduces to $y = N_{L_1/K}(\frac{x_1}{x_2})$. A local solution implying a global solution is equivalent to $\HNP_K(L_1)$. 
\end{proof}
Thus, we can see that the Projective Hasse Norm Principle is a strict generalization of the Hasse Norm Principle.

In this paper, we compare $\PHNP_K$ and $\HNP_K$ across many settings, mainly studying whether it is true that $\PHNP_K(L_1, L_2, \ldots, L_n)$ implies or is implied by $\bigwedge_i \HNP_K(L_i)$ in each setting of the problem. We will show counterexamples for each direction proving that the principles cannot be compared without assumptions  on $L_1,\cdots, L_n$.  In particular, we show in Theorem \ref{abelianfinish} that $\bigwedge_i \HNP_K(L_i)$ implies $\PHNP_K(L_1, L_2, \ldots, L_n)$ when each of the field extensions $L_i/K$ is Galois and abelian, which in turn implies a projective analogue of the Hasse norm theorem. We also show that this result holds when all $L_i/K$ are Galois and $G=\operatorname{Gal}(L_1\cdots L_n/K)$ is dihedral or has fourth-power-free order. We also demonstrate in Proposition \ref{independent} that when $L_1, L_2, \ldots, L_n$ are each Galois extensions of $K$ and are jointly linearly disjoint, the projective principle $\PHNP_K(L_1, L_2, \ldots, L_n)$ holds without assuming the individual Hasse norm principles.

\begin{remark}
   Note that when $L/E$ is quadratic, then the study of the projective Hasse norm principle can be related to computing the Tamagawa number of centralizers of  similitude groups, and is relevant to mass formulas for isogeny classes over finite fields, as in \cite{achter}. We will also mention this application in Remark \ref{rmk:elliptic}.
\end{remark}
\section{Notations}\label{sec:notations}
Let $L_1, L_2, \ldots, L_n$ be finite separable extensions of a global field $K$. Let $L = (L_1L_2\cdots L_n)^{\sharp}$ be the Galois closure of the composite field $L_1L_2\cdots L_n$. Let $\tilde{L}$ denote the \'etale algebra $L_1 \oplus L_2 \oplus \cdots \oplus L_n$. Note that this is different from our previous section where $L$ was the \'etale algebra. Let $G$ denote $\Gal(L/K)$. For simplicity, we let $H^i(M)$ denote $H^i(G, M)$ for any $\mathbb{Z}[G]$-module $M$. For any finite group $A$, we let $A^{\lor} = \mathrm{Hom}(A, \mathbb{Q}/\mathbb{Z})$ denote the Pontryagin dual of $A$ and let $A^{\mathrm{ab}}=A/A^{\mathrm{der}}\cong(A^{\lor})^{\lor}$ denote the abelianization of $A$. We furthermore let $A^{\mathrm{der}}$ denote the derived or commutator subgroup of $A$.

For each $i$, we define $G^{(i)} = \Gal(L/L_i)$. If $L_i/K$ is a Galois extension, we define $G_i = \Gal(L_i/K)$ so that $G_i \cong G/G^{(i)}$. If $L_1, L_2, \ldots, L_n$ are jointly linearly disjoint and Galois, then we have $G \cong \prod_i G_i$ and $G^{(i)} \cong \prod_{j \neq i} G_j$ for all $i$. Here jointly linearly disjoint means that the natural restriction map $G\to\prod_iG_i$ is an isomorphism. We use $\Sh^j(M)$ as shorthand for $\Sh^j(L/K,M)$. In finite-group cohomology diagrams, a product indexed by $D$ runs over one representative of each conjugacy class of actual decomposition subgroups in $G$; this is a finite set and gives the same kernel as the product over all places.

Let us state a simple but useful fact about $G^{(i)}$'s which will be useful to establish an analogue of the Hasse norm theorem for the projective norm.
\begin{lemma}
   Assume that each $L_i/K$ is Galois. Then $\bigcap_i G^{(i)}=1$.
   \label{separatedk}
\end{lemma}
\begin{proof}
An element of the intersection fixes every $L_i$, hence their compositum $L=L_1\cdots L_n$, and is therefore the identity.
\end{proof}

\section{Methodology}
\label{sec:methodology}
Classically, the Hasse norm principle associated with a finite separable field extension $L/K$ is related to the cohomology of norm-one tori as follows. Let $\mathrm R$ and $\mathrm R^{(1)}$ denote Weil restriction and the norm-one torus, respectively. The exact sequence
\[1\longrightarrow\mathbf T^1_{L/K}\longrightarrow
 \mathrm R_{L/K}\mathbb G_m\overset{N_{L/K}}{\longrightarrow}
 \mathbb G_m\longrightarrow1\]
and Hilbert's Theorem 90 give
\[H^1(K,\mathbf T^1_{L/K})
 \cong K^\times/N_{L/K}(L^\times).\]
The analogous statement over $K_v$ uses the norm from $L\otimes_KK_v$. Thus the Hasse norm principle is equivalent to injectivity of the localization map on $H^1(K,\mathbf T^1_{L/K})$. We can do something analogous for the projective norm.

Recall that now $\tilde L=\bigoplus_{i=1}^nL_i$, where each $L_i/K$ is a finite separable field extension and $L$ is the Galois closure of their compositum. The projective norm extends to an exact sequence of algebraic tori
\[1\longrightarrow\mathbf T\longrightarrow
 \mathbb G_m\times\prod_i\mathrm R_{L_i/K}\mathbb G_m
 \overset{\mathrm{PN}}{\longrightarrow}\mathbb G_m^n
 \longrightarrow1,\]
where $\mathrm{PN}(s,(x_i)_i)=(s^{-1}N_{L_i/K}(x_i))_i$.
By Hilbert's Theorem 90, the Galois cohomology of this sequence yields the following commutative diagram with exact rows:
\[
\begin{tikzcd}[column sep=small]
 K^\times\times\prod_iL_i^\times \arrow[r,"\mathrm{PN}"] \arrow[d]
  & (K^\times)^n \arrow[r,"\delta"] \arrow[d,"\pi"]
  & H^1(K,\mathbf T) \arrow[r] \arrow[d]
  & 0\\
 \displaystyle\prod_v\left(K_v^\times\times
       \prod_i(L_i\otimes_KK_v)^\times\right)
  \arrow[r,"\gamma"]
  & \displaystyle\prod_v(K_v^\times)^n \arrow[r,"\eta"]
  & \displaystyle\prod_vH^1(K_v,\mathbf T) \arrow[r]
  & 0.
\end{tikzcd}
\]

\begin{prop}
$\PHNP_K(L_1,L_2,\ldots,L_n)$ is true if and only if
\[H^1(K,\mathbf T)\longrightarrow\prod_vH^1(K_v,\mathbf T)\]
is injective.
\end{prop}
\begin{proof}
By exactness, a tuple $\alpha\in(K^\times)^n$ is a global projective norm if and only if $\delta(\alpha)=0$, and it is a projective norm everywhere locally if and only if $\eta(\pi(\alpha))=0$. Since $\delta$ is surjective and the diagram commutes, the assertion is equivalent to injectivity of the displayed localization map.
\end{proof}

We have shown that $\PHNP_K(L_1,\ldots,L_n)$ is true if and only if $\Sh^1(K,\mathbf T)$ is trivial. Since $L$ splits $\mathbf T$, Hilbert's Theorem 90 and inflation--restriction identify $H^1(K,\mathbf T)$ with $H^1(\operatorname{Gal}(L/K),\mathbf T(L))$, and similarly locally. It is known that the Tate--Shafarevich group is finite, but it is hard to compute directly, so we make use of global duality for tori.

Using the functor $\mathbf X^\star$ mapping a torus to its character lattice, we get an easier way to compute, both mathematically and programmatically, the obstruction. The structure of these lattices is closely linked to the Galois group of the splitting field, and often can be expressed explicitly in terms of this Galois group.

\begin{theorem}[Global duality for tori]
\label{duality}
Let $K$ be a global field, let $\mathbf T$ be an algebraic torus defined over $K$ and split over a finite Galois extension $L/K$, and let $\mathbf X^\star(\mathbf T)$ be its character lattice. Then there is a canonical isomorphism
\[\Sh^1(K,\mathbf T)^{\lor}
  \cong\Sh^2(L/K,\mathbf X^\star(\mathbf T)).\]
\end{theorem}
See \cite[\S0]{demarche-wei} for global duality and \cite[Theorem~2.3]{macedo-newton} for the finite-group formulation over number fields. Our approach will therefore be to give an explicit description of the character lattice of $\mathbf T$ and compute $\Sh^2(L/K,\mathbf X^\star(\mathbf T))$.

\section{Preliminary Results}

\label{sec:elementary}
Though closely related, neither $\HNP$ nor $\PHNP$ directly imply one another, so current results do not extend to the $\PHNP$ condition easily. Thus, we aim to draw connections between the conditions. It turns out that in some cases, we can directly construct a global solution from the local solutions of $\PHNP$ when $\HNP$ holds.

\begin{definition}[$v$-local solution, scale factor] Fix a place $v$ of $K$.
For a given $n$-tuple $(k_1,k_2,\ldots,k_n)\in(K^\times)^n$, if there exist $\ell_i\in(L_i\otimes_KK_v)^\times$ and a constant $s_v\in K_v^\times$ such that $N_{L_i\otimes_KK_v/K_v}(\ell_i)=s_vk_i$ for every $i$, we call $(\ell_1,\ell_2,\ldots,\ell_n;s_v)$ a $v$\emph{-local solution} (relative to $k_1,\ldots,k_n$) with \emph{scale factor} $s_v$.
\end{definition}

Note that the $\PHNP$ is equivalent to the assertion that if there exists a $v$-local solution for every $v$, then there exists a global solution.

\begin{prop}
Suppose that $\HNP_K(L_i)$ holds for $i=1,\ldots,n$ and let $(k_1,\ldots,k_n)\in(K^\times)^n$. If some $s\in K^\times$ is a scale factor for a local solution at every place, then there exists a global solution.
\end{prop}
\begin{proof}
For every $i$, the element $sk_i$ is a local norm at every place. By $\HNP_K(L_i)$ there is $\ell_i\in L_i^\times$ with $N_{L_i/K}(\ell_i)=sk_i$. Thus $(\ell_1,\ldots,\ell_n;s)$ is a global solution.
\end{proof}

The proposition below further shows that we can always reduce the projective Hasse norm principle to the case where all extensions $L_1/K, L_2/K, \ldots, L_n/K$ have degree greater than $1$.

\begin{prop}\label{outfilter} Let $n\geq 1$. 
We have: $$\PHNP_K(L_1, L_2, \ldots, L_n, K) = \PHNP_K(L_1, L_2, \ldots, L_n).$$
\end{prop}
\begin{proof} $(\Leftarrow)$ Assume $\PHNP_K(L_1, L_2, \ldots, L_n)$ holds. Let $k_1,\ldots, k_{n+1}\in K^\times$ such that $v$-local solutions exist for all $v$, relative to the extensions $L_1, \ldots, L_n, K$. This implies that $v$-local solutions of $(k_1,\ldots, k_n)$ exist for all $v$, relative to $L_1,\ldots, L_n$. We can use  $\PHNP_K(L_1, L_2, \ldots, L_n)$ to obtain $(\ell_1,\ldots, \ell_n)\in L_1\times\ldots\times L_n$ and $s\in K^\times$ so that $s k_i = N_{L_i/K}(\ell_i)$ for all $1\leq i\leq n$. Since $N_{K/K}=\mathrm{id}$ we get that $(\ell_1,\ldots, \ell_n,sk_{n+1})$ is a global solution for $(k_1,\ldots, k_{n+1})$, relative to the extensions  $(L_1, L_2, \ldots, L_n, K)$.

$(\Rightarrow)$ Now assume that $\PHNP_K(L_1, L_2, \ldots, L_n, K)$ holds. If $k_1,\ldots, k_n\in K^\times$ have $v$-local solutions for all $v$, with respect to $L_1,\ldots, L_n$, then $k_1,\ldots, k_n,1$ has $v$-local solutions for all $v$ with respect to $L_1,\ldots, L_n, K$, then we can use $\PHNP_K(L_1, L_2, \ldots, L_n, K)$ to get a global solution and restrict it to its first factors to get a global solution relative to $k_1,\ldots,k_n$ and the extensions $L_1,\ldots,L_n$.
This finishes the proof.
\end{proof}

Next, we introduce another reduction of $\PHNP$ to $\HNP$ when $L_1, L_2, \ldots, L_n$ are identical.

\begin{prop}
\label{same}
If $L_1 = L_2 = \cdots = L_n$ for $n \geq 2$, then we must have \[\PHNP_K(L_1, L_2, \ldots, L_n) = \HNP_K(L_1).\]
\end{prop}
\begin{proof} $(\Rightarrow)$
Suppose $\PHNP_K(L_1, L_2, \ldots, L_n)$ holds. Let  $k \in K^\times$ such that $k$ belongs in the image of the norm of $L_1/K$ locally. Since $1$ is always in the image of the norm, there exists a global solution to $(k, 1, 1, \ldots, 1)$ with respect to $L_1,\ldots, L_n$. Therefore, there are $s\in K^\times$ and  $(\ell_1,\ldots \ell_n)\in L_1^\times\times\cdots\times L_n^\times$ so that $N_{L_1/K}(\ell_1) = sk$ and $N_{L_i/K}(\ell_i) = N_{L_1/K}(\ell_i) = s\cdot 1$ for $i\geq2$. Therefore, $k= N_{L_1/K}(\ell_1/\ell_2)$. This implies that $\HNP_K(L_1)$ holds.

$(\Leftarrow)$ Assume that $ \HNP_K(L_1)$ holds. Assume that $k_1,\ldots, k_n\in K^\times$ locally scale to a norm of $L_1^\times\times\cdots\times L_n^\times$. This means that for every place $v$ there are $s_v\in K_v^\times$ and $\ell_1,\ldots, \ell_n\in(L_1\otimes_KK_v)^\times$ so that $N_{L_1/K}(\ell_i) = s_vk_i$ for all $i\leq n$. This implies that $k_i/k_n = N_{L_1/K}(\ell_i/\ell_n)$. Using the Hasse norm principle, there are $x_1,\ldots, x_{n-1}\in L_1^\times$ so that $N_{L_1/K}(x_i) = k_i/k_n$. Let $x_n = 1$. Then for all $1\leq i\leq n$ we have $N_{L_1/K}(x_i) = \frac{1}{k_n} k_i$, so we can take $\frac{1}{k_n}$ as global scale factor.
\end{proof}

\begin{remark}
    \label{rmk:elliptic}
Although the previous proposition examines a restricted setting, equal-norm tori also have an interesting application. In \cite{achter}, the authors establish a mass formula for isogeny classes of principally polarized abelian varieties over finite fields, weighted by the reciprocals of the sizes of their automorphism groups. Their hypotheses include a commutative endomorphism ring and the assumption that the abelian variety is ordinary or that the base field is prime. The formula involves the Tamagawa number of the torus centralizing Frobenius on the rational Tate module.

This Tamagawa number can be written
\[\tau(\mathbf T)=
 \frac{|H^1(\mathbb Q,\mathbf X^*(\mathbf T))|}
      {|\Sh^1(\mathbb Q,\mathbf T)|},\]
and the difficulty in its calculation is often to determine the denominator.

For a regular semisimple element of $\mathrm{GL}_2$, the centralizer is either split or a restriction of scalars and has Tamagawa number $1$. For a regular semisimple element of $\mathrm{GSp}_4$ with two distinct irreducible quadratic factors, the centralizer is an equal-norm torus inside $(\mathrm{GL}_2\times\mathrm{GL}_2)^0$, the group of pairs $(g_1,g_2)$ with matching determinants. Its Tate--Shafarevich group is exactly the obstruction to $\PHNP_{\mathbb Q}(L_1,L_2)$ for the two quadratic fields. Thus our results determine the denominator in this case. Repeated Frobenius factors need not have torus centralizers, so the preceding proposition alone does not give a mass formula for repeated elliptic-curve factors.
\end{remark}

\section{Description of Character Lattices}
\label{sec:description}
Recall that we have the following exact sequence between algebraic tori by mapping each element of $\mathbf{T}$ to its common value under the maps $N_{L_i/K}$:

\[\begin{tikzcd}
1 \arrow[r] & \prod_i R^{(1)}_{L_i/K} \mathbb{G}_m \arrow[r] & \mathbf{T} \arrow[r] & \mathbb{G}_m \arrow[r] & 1.
\end{tikzcd}\]

Let $\Lambda := \mathbf{X}^*(\mathbf{T})$ and $\Lambda^1 :=  \mathbf{X}^*(\prod_i R^{(1)}_{L_i/K} \mathbb{G}_m)$. From this, we can obtain a short exact sequence of character lattices by taking the dual of the previous exact sequence:

\begin{equation}\begin{tikzcd}
0 \arrow[r] & \mathbb{Z} \arrow[r] & \Lambda \arrow[r] & \Lambda^1 \arrow[r] & 0.
\end{tikzcd}  \label{shortZ}\end{equation}

We may now describe the character lattice of $\mathbf{T}$.

\begin{prop}
    Let $S_i$ denote the left coset space $G/G^{(i)}$ equipped with a left $G$-action. Let $d_i = \sum_{g \in S_i}g \in \mathbb{Z}[S_i]$ for each $i$. Then it follows that $$\mathbf{X}^*(\mathbf{T}) \cong \left(\prod_i \mathbb{Z}[S_i] \right)/\left\lbrace(\lambda_1d_1, \lambda_2d_2, \ldots, \lambda_nd_n) : \lambda_i \in \mathbb{Z}, \sum_i \lambda_i = 0\right\rbrace.$$
    \label{fullchar}
\end{prop}
\begin{proof}

    Let $\mathbf{T}^1=\prod_i\mathrm R^{(1)}_{L_i/K}\mathbb G_m$. The embeddings $\mathbf{T}^1(L)\subset \mathbf{T}(L)\subset \mathrm{R}_{\tilde L/K}\mathbb{G}_m(L)$ yield the surjections \[\bigoplus_{i}\mathbb{Z}[S_i]=\mathbf{X}^\star(\mathrm{R}_{\tilde L/K}\mathbb{G}_m)\overset{\varphi}{\rightarrow}\mathbf{X}^\star(\mathbf{T})\rightarrow\mathbf{X}^\star(\mathbf{T}^1) =\bigoplus_{i}\mathbb{Z}[S_i]/\langle d_i\rangle. \]
    
    Note that  $d_i$ corresponds to the norm map, viewed as a character, and since elements of $\mathbf{T}(L)$ are tuples $\ell = (\ell_1,\ldots, \ell_n)$ with matching norm $\eta$, we get that an element $\lambda_1 d_1\oplus \cdots \oplus \lambda_n d_n$ maps $\ell$ to 
    \[\prod_{i=1}^n N_{L_i/K}(\ell_i)^{\lambda_i} = \eta^{\sum_{i=1}^n \lambda_i}.\]
We just showed that $\{(\lambda_1d_1, \lambda_2d_2, \ldots, \lambda_nd_n) : \lambda_i \in \mathbb{Z}, \sum_i \lambda_i = 0\}$ is in the kernel of the surjection $\varphi$. Each $d_i$ is primitive in its permutation lattice. Extending each $d_i$ to an integral basis shows that the displayed relation lattice is saturated: the quotient of $\bigoplus_i\mathbb Z d_i$ by the relations $\sum_i\lambda_i=0$ is $\mathbb Z$. Its rank is $n-1$, which is also the rank of $\ker\varphi$. Therefore this saturated sublattice is exactly $\ker\varphi$.
\end{proof}
In the particular case where each extension is Galois, we get the following. 
\begin{cor} 
\label{charstructure}
    Suppose $L_1, L_2, \ldots, L_n$ are all Galois over $K$. Let $d_i = \sum_{g \in G_i} g \in \mathbb{Z}[G_i]$ for all $i$. Then $ \mathbf{X}^*(\mathbf{T}) \cong (\prod_i \mathbb{Z}[G_i])/\{(\lambda_1 d_1, \lambda_2 d_2, \ldots, \lambda_n d_n) : \lambda_i \in \mathbb{Z}, \sum_i \lambda_i = 0\}$. 
\end{cor}
\begin{proof}
    When $L_i$ is Galois, then $S_i = G_i$ is a group. 
\end{proof}

This description lets us use SageMath to build the character lattice, using the package for algebraic tori implemented by the second author in \cite{thomaspaper}. In the standard permutation bases of the induced lattices, the following code uses the two orbit sums explicitly:
\begin{lstlisting}
def MakeLattice(G, H1, H2):
    IL1 = GLattice(H1, 1).induced_lattice(G)
    IL2 = GLattice(H2, 1).induced_lattice(G)
    IL = IL1.direct_sum(IL2)
    n1, n2 = IL1.rank(), IL2.rank()
    a = vector(ZZ, [1]*n1 + [0]*n2)
    b = vector(ZZ, [0]*n1 + [1]*n2)
    HNPLattice = IL.quotient_lattice(IL.sublattice([a, b]))
    PHNPLattice = IL.quotient_lattice(IL.sublattice([a-b]))
    return [PHNPLattice, HNPLattice]
\end{lstlisting}
This construction gives the pair $\left(\mathbf X^\star(\mathbf T),\bigoplus_i\mathbf X^\star(\mathbf T^1_{L_i/K})\right)$. The induced lattices and their direct sum must use these standard permutation coordinates; the package version and basis conventions should be recorded with any numerical output. The proofs below do not rely on this code.

This description allows us to solve some cases algorithmically. 
\begin{prop}
    If $L_1, L_2$ are quadratic extensions, then $\PHNP_K(L_1, L_2)$ holds.
\end{prop}
\begin{proof}
If $L_1\ne L_2$, the two quadratic extensions are linearly disjoint. The theorem of Pollio and Rapinchuk in \S\ref{sec:statement}, together with Remark~\ref{identicalpair}, proves the result. If $L_1=L_2$, use Proposition~\ref{same} and the Hasse norm theorem.
\end{proof}

Notably in the above proposition, all quadratic extensions are Galois and have abelian Galois groups. In later sections, we generalize this result.
\section{Counterexample to \texorpdfstring{$\PHNP \Longrightarrow \HNP$}{PHNP->HNP}}
\label{sec:counterexamplephnptohnp}
Though the two conditions agree in some elementary cases, they are not equivalent in general. The following example uses a biquadratic extension.

\begin{example}
In particular, we investigate the case where $L_1\subset L_2$, with $L_1/K$ quadratic and $L_2/K$ biquadratic. By transitivity of norms,
\[N_{L_2/K}(L_2^\times)\subseteq N_{L_1/K}(L_1^\times),\]
and the analogous inclusion holds locally. Remark~\ref{identicalpair} therefore gives
\[\PHNP_K(L_1,L_2)=\HNP_K(L_1),\]
which holds by the Hasse norm theorem. This conclusion does not require any restriction on the decomposition groups of $L_2/K$.

Now, choosing $L_1=\mathbb Q(\sqrt{-3})$ and $L_2=\mathbb Q(\sqrt{-3},\sqrt{13})$ gives the desired example. The field $L_2$ does not satisfy the Hasse norm principle: $3$ is everywhere locally a norm but is not a global norm; see \cite{hasse} and \cite[Introduction]{macedo-newton}. Thus $\PHNP$ does not imply $\HNP$ in each of the constituent fields.

Similarly, we can take $L_1=\mathbb Q(\sqrt p)$ and $L_2=\mathbb Q(\sqrt p,\sqrt q)$, where $p\ne q$ are primes such that $\left(\frac pq\right)=1$ and $p\equiv q\equiv1\pmod4$. The decomposition groups of $L_2/\mathbb Q$ are then cyclic. Tate's description of the norm obstruction gives $\Sh(L_2/\mathbb Q)\cong\mathbb Z/2\mathbb Z$, since $H^3((\mathbb Z/2\mathbb Z)^2,\mathbb Z)\cong\mathbb Z/2\mathbb Z$ and $H^3(D,\mathbb Z)=0$ for cyclic $D$; see \cite[Theorem~2.3]{macedo-newton}. Thus these pairs also satisfy $\PHNP$ while $\HNP_{\mathbb Q}(L_2)$ fails.
\end{example}

\section{A Condition for \texorpdfstring{$\HNP \Longrightarrow \PHNP$}{HNP->PHNP}}
\label{sec:condition}
In this section, we find sufficient conditions for identifying cases where $\HNP$ in all of the constituent fields implies $\PHNP$ based on work in \cite{thomaspaper} and \cite{liang2022tamagawa}.

\begin{prop}
\label{subsetmiddle}
    Suppose $\bigwedge_i \HNP_K(L_i)$ holds. Let $x$ be any element of $\Sh^2(\Lambda) \subseteq H^2(\Lambda)$, and let $\phi$ be the canonical map from $H^2(\Lambda)$ to $H^2(\Lambda^1)$ defined from equation (\ref{shortZ}). Then $x \in \mathrm{Ker}(\phi)$. 

\end{prop}

\begin{proof}
    We begin by taking the cohomology of the exact sequence in equation (\ref{shortZ}). We get

    \[\begin{tikzcd}
\cdots \arrow[r] & H^2(\mathbb{Z}) \arrow[r] & H^2(\Lambda) \arrow[r] & H^2(\Lambda^1) \arrow[r] & \cdots.
\end{tikzcd}\]
Consider the following induced diagram.

\[\begin{tikzcd}
& 0 \arrow[d] & 0 \arrow[d] & 0 \arrow[d] & \\
& \Sh^2(\mathbb{Z}) \arrow[d]  & \Sh^2(\Lambda) \arrow[d, "\iota"] & \Sh^2(\Lambda^1) \arrow[d] & \\
\cdots \arrow[r] & H^2(\mathbb{Z}) \arrow[d] \arrow[r, "e"] & H^2(\Lambda) \arrow[r, "\phi"] \arrow[d, "a"] & H^2(\Lambda^1) \arrow[r] \arrow[d, "d"] & \cdots \\
\cdots \arrow[r] & \prod H^2(D, \mathbb{Z}) \arrow[r] & \prod H^2(D, \Lambda) \arrow[r, "b"] & \prod H^2(D, \Lambda^1) \arrow[r] & \cdots
\end{tikzcd}\]

By global duality and the hypothesis, $\Sh^2(\Lambda^1)=0$, so $d$ is injective. The image $\phi(x)$ is locally trivial because the diagram commutes. Hence $d(\phi(x))=0$, and injectivity gives $\phi(x)=0$ as desired. \end{proof}

Notice furthermore that $\mathrm{Im}\  e = \mathrm{Ker}\  \phi$. Therefore if $\HNP_K(L_i)$ holds for every $i$, equivalently if $d$ is injective, we have that $\mathrm{Im}\  \iota \subseteq \mathrm{Im}\  e$. Note that if $e$ is trivial, it follows that $\Sh^2(\Lambda)$ is trivial, so $\PHNP$ must hold with the assumption that $\bigwedge_i \HNP_K(L_i)$ is true. Following the notations of \cite{liang2022tamagawa}, let $H^2(\mathbb{Z})' := \{x \in H^2(\mathbb{Z}) : e(x) \in \mathrm{Im}\  \iota\}$. We can now extract a set of sufficient conditions for when $\HNP$ implies $\PHNP$ using decomposition groups.

\begin{prop}\label{descriptor}
    Suppose $\bigwedge_i \HNP_K(L_i)$ holds. Let $\psi$ be the map from $\prod H^1(D, \Lambda^1)$ to $\prod H^2(D, \mathbb{Z})$ in the localized long exact sequence of (\ref{shortZ}), with $D$ ranging over the finite set fixed in \S\ref{sec:notations}. Then $e$ induces a canonical isomorphism
\[\Sh^2(\Lambda)\cong H^2(\mathbb Z)'/\ker e.\]
In particular, $|\Sh^2(\Lambda)|=|H^2(\mathbb Z)'|/|\ker e|$. In particular, if $|\mathrm{Im}\ \psi| = |\mathrm{Ker}\  e|$ then $\PHNP$ holds.
\end{prop}
\begin{proof}
    First, we show that $\Sh^2(\mathbb{Z}) = 0$. Suppose otherwise, then there exists a nonzero character $f \in H^2(G, \mathbb{Z}) \cong G^{\lor}$ in $\Sh^2(\mathbb{Z})$. In particular, there exists some $g \in G$ such that $f(g) \neq 0$. Then $\langle g \rangle \subset G$ is a cyclic subgroup of $G$, so by the Chebotarev density theorem (see \cite{milne1997class}, p.164), it follows that $\langle g \rangle$ is conjugate to a decomposition group. The character $f$ is invariant under conjugation. However $f|_{\langle g \rangle} \neq 0$, giving contradiction.

    Now, we can redraw the commutative diagram from Proposition \ref{subsetmiddle}, shifting the sequence one term to the left and labeling functions $e, f, g, a, \xi$ as shown.

    \[\begin{tikzcd}
& 0 \arrow[d] & & 0 \arrow[d] & \\
& \Sh^1(\Lambda^1) \arrow[d]  & 0 \arrow[d] & \Sh^2(\Lambda) \arrow[d, "\iota"] & \\
\cdots \arrow[r] & H^1(\Lambda^1) \arrow[d] \arrow[r, "\xi"] & H^2(\mathbb{Z}) \arrow[r, "e"] \arrow[d, "f"] & H^2(\Lambda) \arrow[r] \arrow[d, "a"] & \cdots \\
\cdots \arrow[r] & \prod H^1(D, \Lambda^1) \arrow[r, "\psi"] & \prod H^2(D, \mathbb{Z}) \arrow[r, "g"] & \prod H^2(D, \Lambda) \arrow[r] & \cdots
\end{tikzcd}\]

In particular, since $\iota$ is injective, the image of $H^2(\mathbb{Z})'$ under $e$ is exactly $\mathrm{Im}\ \iota$. The kernel of this restriction of $e$ is $\ker e$, so the first isomorphism theorem gives the displayed isomorphism and its cardinality formula. 

Furthermore, for any element $x$ of $H^2(\mathbb{Z})'$, $e(x) \in \mathrm{Im}\ \iota$ so $a(e(x)) = 0$. Thus $f(x) \in \mathrm{Ker}\ g$ so $f(x) \in \mathrm{Im}\ \psi$. Thus since $f$ is injective, it follows that $|H^2(\mathbb{Z})'| \leq |\mathrm{Im}\ \psi|$ so if $|\mathrm{Im}\  \psi| = |\mathrm{Ker}\  e|$ then $\displaystyle 1 = \frac{|\mathrm{Im}\  \psi|}{|\mathrm{Ker}\  e|} \geq \frac{|H^2(\mathbb{Z})'|}{|\mathrm{Ker}\  e|} = |\Sh^2(\Lambda)|$ as desired. \end{proof}

\begin{cor}
If $\mathrm{Ker}\ e=H^2(\mathbb Z)\cong G^{\lor}$, then
\[\bigwedge_i\HNP_K(L_i)\Longrightarrow\PHNP_K(L_1,L_2,\ldots,L_n).\]
\end{cor}
\begin{proof}
    It follows from Proposition \ref{descriptor} that $1 \geq \frac{|H^2(\mathbb{Z})'|}{|H^2(\mathbb{Z})|} = |\Sh^2(\Lambda)|$, so $\Sh^2(\Lambda) = 0$ as desired.
\end{proof}

Now that we have a path to determining when $\PHNP$ holds in a given collection of field extensions, we hope to understand the map $e$ so that we can explicitly compute its kernel and $H^2(\mathbb Z)'$. We proceed by analyzing the map $\xi$.

\begin{prop}
    Let $\Lambda^1_i = \mathrm{Ind}^G_{G^{(i)}}\mathbb{Z}/\langle d_i \rangle$ for any $i$. It follows that $H^1(G, \Lambda_i^1) \cong \{f\in G^{\lor} : f|_{G^{(i)}} \equiv 0 \}$.
    \label{prexi}
\end{prop}

\begin{proof}

Recall from Proposition \ref{fullchar} that 
\begin{align*}
    \Lambda^1 &= \bigoplus \limits_{i = 1}^n \mathrm{Ind}^G_{G^{(i)}}\mathbb{Z}/\langle d_i \rangle \\
    &\cong \bigoplus \limits_{i = 1}^n \mathbb{Z}[G/G^{(i)}]/\left\langle \sum_{g \in G/G^{(i)}} g\right\rangle.
\end{align*}

Thus we have $H^j(G,\Lambda^1)=\bigoplus_i H^j(G,\Lambda_i^1)$ for every $j$.

Letting $\mathrm{diag}_i$ denote the map from $k \in \mathbb{Z}$ to $kd_i$, we obtain the following short exact sequence:

\[\begin{tikzcd}
0 \arrow[r] & \mathbb{Z} \arrow[r, "\mathrm{diag}_i"] & \mathrm{Ind}^G_{G^{(i)}}\mathbb{Z} \arrow[r] & \Lambda^1_i \arrow[r] & 0.
\end{tikzcd}\]

Computing the cohomology, we obtain:

\[\begin{tikzcd}[cramped, sep=small]
\cdots \arrow[r] & H^1(G, \mathrm{Ind}^G_{G^{(i)}}\mathbb{Z})  \arrow[r] & H^1(G, \Lambda^1_i) \arrow[r] & H^2(G, \mathbb{Z}) \arrow[r] & H^2(G, \mathrm{Ind}^G_{G^{(i)}}\mathbb{Z}) \arrow[r] & \cdots.
\end{tikzcd}\]

Note that $H^1(G, \mathrm{Ind}^G_{G^{(i)}}\mathbb{Z}) \cong H^1(G^{(i)}, \mathbb{Z}) = 0$ by Shapiro's lemma and the equality $\operatorname{Hom}(G^{(i)},\mathbb Z)=0$, $H^2(G, \mathbb{Z}) = G^{\lor}$, and $H^2(G, \mathrm{Ind}^G_{G^{(i)}}\mathbb{Z}) \cong H^2(G^{(i)}, \mathbb{Z}) = (G^{(i)})^{\lor}$. Rewriting the cohomology, we have

\[\begin{tikzcd}
0 \arrow[r] & H^1(G, \Lambda^1_i) \arrow[r] & G^{\lor} \arrow[r, "r_i"] & (G^{(i)})^{\lor} \arrow[r] & \cdots.
\end{tikzcd}\]

It follows that $H^1(G, \Lambda^1_i) \cong \mathrm{Ker}\ r_i = \{f\in G^{\lor} : f|_{G^{(i)}} \equiv 0 \}$ as desired.\end{proof}

The sequence (\ref{shortZ}) is the pushout of the direct sum of the preceding sequences along the sum map $\mathbb Z^n\to\mathbb Z$. Naturality of the connecting homomorphisms therefore shows that $\xi$ sends $(f_1,\ldots,f_n)\in\bigoplus_i H^1(\Lambda_i^1)$ to $\sum_i f_i\in G^{\lor}$. Hence we obtain the following proposition.

\begin{prop}
    \label{xistructure}
    We have 
    $$\mathrm{Ker}\ e = \mathrm{Im}\ \xi = \left\lbrace f \in G^{\lor} : \exists f_1,\ldots,f_n\in G^{\lor},\ f=\sum_i f_i,\ f_i|_{G^{(i)}}=0\text{ for all }i\right\rbrace$$ where $f_1, f_2, \ldots, f_n \in G^{\lor} $.
\end{prop}

\begin{remark}
Both Propositions~\ref{prexi} and~\ref{xistructure} can be adapted to a subgroup $D\subseteq G$, but for non-Galois $L_i/K$ one must use all the double cosets $D\backslash G/G^{(i)}$, not only $D\cap G^{(i)}$. The lattice $\Lambda_i^1$ is restricted to $D$; it is not redefined.
\end{remark}

\begin{prop}
\label{psistructure}
Let $e_D$ be the map
\[H^2(D,\mathbb Z)\longrightarrow H^2(D,\Lambda)\]
induced by (\ref{shortZ}). Under $H^2(D,\mathbb Z)\cong D^{\lor}$, its kernel consists of the sums
\[f=\sum_{i=1}^n f_i,\qquad f_i\in D^{\lor},\qquad
 f_i|_{D\cap gG^{(i)}g^{-1}}=0
 \quad\text{for every }i\text{ and }g\in G.\]
If every $L_i/K$ is Galois, the condition reduces to $f_i|_{D\cap G^{(i)}}=0$.
\end{prop}
\begin{proof}
Choose representatives for $D\backslash G/G^{(i)}$. Mackey's decomposition gives
\[\operatorname{Res}^G_D\mathbb Z[G/G^{(i)}]
 \cong\bigoplus_{g\in D\backslash G/G^{(i)}}
       \mathbb Z[D/(D\cap gG^{(i)}g^{-1})].\]
When $D$ is the decomposition group at a place $w$ of $L$ above $v$, these double cosets also give
\[L_i\otimes_KK_v\cong
  \prod_{g\in D\backslash G/G^{(i)}}L_w^{D\cap gG^{(i)}g^{-1}}.\]
Under Mackey's decomposition, $d_i$ is the tuple of the orbit sums. Restricting $0\to\mathbb Z\to\mathbb Z[G/G^{(i)}]\to\Lambda_i^1\to0$ therefore retains a single diagonal copy of $\mathbb Z$ across all these summands. By Shapiro's lemma, its long exact sequence identifies $H^1(D,\Lambda_i^1)$ with
\[\ker\left(D^{\lor}\longrightarrow
 \bigoplus_{g\in D\backslash G/G^{(i)}}
 (D\cap gG^{(i)}g^{-1})^{\lor}\right),\]
where the components are restriction maps. Thus it is the group of characters vanishing on every displayed stabilizer. This condition is independent of the double-coset representatives, since characters are invariant under conjugation in $D$, and it is equivalent to testing every $g\in G$.

The connecting map $H^1(D,\Lambda^1)\to H^2(D,\mathbb Z)$ is again the sum of the component connecting maps. Exactness gives the formula for $\ker e_D$. If $G^{(i)}$ is normal, all its conjugates equal $G^{(i)}$, giving the last assertion.
\end{proof}

Using the above result, we can also deduce an explicit formulation for $H^2(\mathbb{Z})'$ analogous to that of Proposition \ref{xistructure} for $\mathrm{Ker}\ e$. 

\begin{prop}
\label{numerator}
The set $H^2(\mathbb Z)'\subseteq H^2(\mathbb Z)\cong G^{\lor}$ is exactly
\[\{f\in G^{\lor}:f|_D\in\ker e_D
       \text{ for every decomposition group }D\}.\]
Here each local kernel is given by Proposition~\ref{psistructure}, including its full double-coset condition when $L_i/K$ is not Galois. If $\bigwedge_i\HNP_K(L_i)$ holds, then
\[\Sh^2(\Lambda)\cong H^2(\mathbb Z)'/\ker e,\]
and hence $\PHNP_K(L_1,\ldots,L_n)$ holds if and only if $H^2(\mathbb Z)'=\ker e$.
\end{prop}
\begin{proof}
By definition, $f\in H^2(\mathbb Z)'$ exactly when $e(f)$ is locally trivial. Naturality gives $\operatorname{res}_D(e(f))=e_D(f|_D)$, proving the description. The isomorphism and the final assertion follow from Proposition~\ref{descriptor} and global duality.
\end{proof}

\begin{prop}
\label{independent}
If $L_1,\ldots,L_n$ are finite Galois extensions of a global field $K$ and the natural map
\[G\longrightarrow\prod_iG_i\]
is an isomorphism, then $\PHNP_K(L_1,\ldots,L_n)$ holds. No individual Hasse norm principle is assumed.
\end{prop}
\begin{proof}
Identify $G=\prod_iG_i$ and $G^{(i)}=\prod_{j\ne i}G_j$. Let $x\in\Sh^2(\Lambda)$, and write $(x_i)_i$ for its image in $H^2(G,\Lambda^1)=\bigoplus_iH^2(G,\Lambda_i^1)$.

Each $x_i$ restricts to zero on every cyclic subgroup, by Chebotarev. Since $G^{(i)}$ acts trivially on the lattice $\Lambda_i^1$, classes in $H^2(G^{(i)},\Lambda_i^1)$ are detected by restriction to cyclic subgroups. Hence $x_i|_{G^{(i)}}=0$. Also $H^1(G^{(i)},\Lambda_i^1)=0$, so inflation--restriction shows that $x_i$ is inflated from a class $y_i\in H^2(G_i,\Lambda_i^1)$.

The sequence
\[0\longrightarrow\mathbb Z\longrightarrow\mathbb Z[G_i]
 \longrightarrow\Lambda_i^1\longrightarrow0\]
gives a connecting isomorphism
\[H^2(G_i,\Lambda_i^1)\cong H^3(G_i,\mathbb Z),\]
since the positive-degree cohomology of $\mathbb Z[G_i]$ vanishes. By exactness of (\ref{shortZ}), the image of $(x_i)_i$ in $H^3(G,\mathbb Z)$ is zero. The connecting map is the sum of the individual connecting maps and commutes with inflation. Restriction to the coordinate subgroup $G_i$ kills the inflated terms from all $G_j$ with $j\ne i$. Thus the image of $y_i$ under the displayed isomorphism is zero. Consequently all $y_i$, and hence all $x_i$, are zero.

It follows that $x\in\operatorname{im}e$. Every character of $G=\prod_iG_i$ is a sum of coordinate characters, with its $i$th summand trivial on $G^{(i)}$. Proposition~\ref{xistructure} therefore gives $\ker e=G^{\lor}$, so $e=0$ and $x=0$. Global duality proves the result.
\end{proof}

The same proof shows that restriction to cyclic subgroups is injective on $H^2(G,\Lambda)$. In particular, the equal-norm torus $\mathbf T$ satisfies weak approximation, by the usual exact sequence for tori; see \cite[\S0]{demarche-wei}.

\section{Counterexample to \texorpdfstring{$\HNP \Longrightarrow \PHNP$}{HNP->PHNP}}
\label{sec:counterexamplehnptophnp}
In this section, we construct a counterexample to the claim that $\bigwedge_i\HNP_K(L_i)\Longrightarrow\PHNP_K(L_1,\ldots,L_n)$, showing that neither principle implies the other. We first give a non-Galois example and then explain a Galois example already present in the multinorm literature.

\begin{example}
\label{s4example}
Let $L$ be the splitting field of $x^4-x+1$ over $K=\mathbb Q$. Its Galois group is $G\cong S_4$; see \cite{lmfdb:s4}. Let $H_1=\langle(12)\rangle$ and $H_2=\langle(1234)\rangle$, and let $L_1=L^{H_1}$ and $L_2=L^{H_2}$. These are extensions of degrees $12$ and $6$, respectively. Both subgroups are core-free, so $L$ is the normal closure of each extension.

The Hasse norm principle holds for both $L_i/\mathbb Q$ by \cite[Proposition~1.8(iii)]{macedo-newton}. Indeed, $H_1$ has even order and contains no double transposition; $H_2$ has even order and its subgroup generated by a double transposition has index $2$. Thus neither subgroup is in the exceptional cases of that proposition. This replaces the need for a computer calculation of the individual norm obstructions.

The only finite ramified prime of $L$ is $229$. At this prime the quartic root field has one unramified and one ramified quadratic factor, so the local splitting field is biquadratic; see \cite{lmfdb:s4}. After relabeling the roots, its decomposition group is
\[D_{229}=\langle(12),(34)\rangle.\]
All other decomposition groups are cyclic, including the group at the real place.

Now consider Proposition~\ref{descriptor}. The character group $G^{\lor}$ consists of the zero character and the sign character, written additively as
\[f(g)=\begin{cases}0&g\in A_4,\\[2pt]\tfrac12&g\notin A_4.\end{cases}\]
Since $f$ is nonzero on both $H_1$ and $H_2$, Proposition~\ref{xistructure} gives $\ker e=0$.

Next, we check the numerator using the full Mackey condition of Proposition~\ref{psistructure}. If a cyclic subgroup $D$ is generated by a transposition, then $D\cap gH_2g^{-1}=1$ for every $g\in G$. Every other cyclic subgroup of $S_4$ contains no transposition, so $D\cap gH_1g^{-1}=1$ for every $g\in G$. In either case we can take one of the two local characters to be $f|_D$ and the other to be zero.

For $D=D_{229}$, the intersections with conjugates of $H_1$ include $\langle(12)\rangle$ and $\langle(34)\rangle$, while every intersection $D\cap gH_2g^{-1}$ is contained in $\langle(12)(34)\rangle$. Since $f((12)(34))=0$, the choice $f_1=0$, $f_2=f|_D$ satisfies all the Mackey conditions. Thus $f|_D\in\ker e_D$ at every place.

It follows that $H^2(\mathbb Z)'=G^{\lor}$, and hence
\[\Sh^2(\Lambda)\cong H^2(\mathbb Z)'/\ker e
 \cong\mathbb Z/2\mathbb Z.\]
Therefore $\HNP_{\mathbb Q}(L_1)$ and $\HNP_{\mathbb Q}(L_2)$ both hold, but $\PHNP_{\mathbb Q}(L_1,L_2)$ does not.
\end{example}

\begin{example}[Demarche--Wei]
\label{galoiscounterexample}
The implication also fails when both constituent fields are Galois. Let
\[L_1=\mathbb Q(i,\sqrt[4]{2}),\qquad
  L_2=\mathbb Q(i,\sqrt{7\sqrt2}).\]
These are Galois extensions of degree $8$. In the proof of \cite[Proposition~12]{demarche-wei}, it is checked that $L_1\otimes_{\mathbb Q}\mathbb Q_2$ and $L_2\otimes_{\mathbb Q}\mathbb Q_2$ are the same field of degree $8$ over $\mathbb Q_2$. Thus the decomposition group at $2$ is the whole Galois group for each $L_i/\mathbb Q$. Its restriction map is the identity in the finite-group description of the norm obstruction, so $\HNP_{\mathbb Q}(L_i)$ holds for both $i$.

On the other hand, \cite[Example~13]{demarche-wei} shows that
\[N_{L_1/\mathbb Q}(x_1)N_{L_2/\mathbb Q}(x_2)=97\]
is soluble over every completion but not over $\mathbb Q$. By Remark~\ref{identicalpair}, this gives a failure of $\PHNP_{\mathbb Q}(L_1,L_2)$; equivalently, the tuple $(97,1)$ has local common-scale solutions but no global one. Proposition~12 of the same paper computes the obstruction as $\mathbb Z/2\mathbb Z$.

Their compositum is
\[L=L_1L_2=\mathbb Q(i,\sqrt[4]{2},\sqrt7),\qquad
  \operatorname{Gal}(L/\mathbb Q)\cong D_4\times\mathbb Z/2\mathbb Z,\]
where $D_4$ is dihedral of order $8$. Indeed, $\mathbb Q(\sqrt7)$ is not contained in $L_1$, since it ramifies at $7$ whereas $L_1$ ramifies only at $2$. In particular, $[L:\mathbb Q]=16$.
\end{example}

We now know that $\PHNP$ holds for jointly linearly disjoint Galois extensions, but the implication $\HNP\Longrightarrow\PHNP$ fails in general even for Galois extensions. We next give sufficient conditions on the Galois group under which the implication does hold.

\section{Studying \texorpdfstring{$\HNP \Longrightarrow \PHNP$}{HNP->PHNP} in the Galois Case}
\label{sec:galois}
We proceed to derive a set of sufficient conditions that allow us to prove $\HNP \Longrightarrow \PHNP$ using only the Galois groups of the fields $L_1, L_2, \ldots, L_n$. In this section, we exclusively consider the case where each $L_i/K$ is a Galois extension.

\begin{prop}
    If for any $f \in G^{\lor}$, there exist characters $f_i \in G^{\lor}$ for each $i$ such that $f_i |_{G^{(i)}} \equiv 0$ and such that $f = \sum_i f_i$, then $\bigwedge_i \HNP_K(L_i) \Longrightarrow \PHNP_K(L_1, L_2, \ldots, L_n)$.
    \label{weaker}
\end{prop}

\begin{proof}

First, recall from Proposition \ref{xistructure} that $$\mathrm{Ker}\ e = \{f \in G^{\lor} : \exists f_1,\ldots,f_n\in G^{\lor},\ f=\sum_i f_i,\ f_i|_{G^{(i)}}=0\text{ for all }i \}.$$

If it is true that $\mathrm{Ker}\ e = H^2(\mathbb{Z})$, then Proposition \ref{descriptor} implies that $\displaystyle |\Sh^2(\Lambda)| = \frac{|H^2(\mathbb{Z})'|}{|\mathrm{Ker}\ e|} \leq \frac{|H^2(\mathbb{Z})|}{|\mathrm{Ker}\ e|} = 1$, so $\PHNP$ holds. \end{proof}

This weaker condition for $\PHNP$ allows us to study $\HNP \Longrightarrow \PHNP$ using only the structure of the Galois group $G$. The following proposition gives us an algorithmic approach to verify the conditions of Proposition \ref{weaker} in $G$.

\begin{prop}
If $\bigcap_iG^{\mathrm{der}}G^{(i)}=G^{\mathrm{der}}$, then
\[\bigwedge_i\HNP_K(L_i)\Longrightarrow\PHNP_K(L_1,\ldots,L_n).\]
\label{concretederived}
\end{prop}
\begin{proof}
Since each $G^{(i)}$ is normal, $(G/G^{(i)})^{\mathrm{ab}}\cong G/(G^{\mathrm{der}}G^{(i)})$. The kernel of the natural map
\[\alpha:G^{\mathrm{ab}}\longrightarrow
       \prod_i(G/G^{(i)})^{\mathrm{ab}}\]
is $(\bigcap_iG^{\mathrm{der}}G^{(i)})/G^{\mathrm{der}}$. Thus the hypothesis says exactly that $\alpha$ is injective. Its dual is the sum of the inflated character maps
\[\prod_i(G/G^{(i)})^{\lor}\longrightarrow G^{\lor},\]
and is therefore surjective. Its image is $\ker e$ by Proposition~\ref{xistructure}. Proposition~\ref{weaker} now gives the result.
\end{proof}

\subsection{Special Cases where HNP \texorpdfstring{$\implies$}{->} PHNP}

We use Proposition \ref{concretederived} to prove that $\HNP \Longrightarrow \PHNP$ in several classes of $G$, including when $G$ is abelian or dihedral.

\begin{theorem}
    If $L_i/K$ is Galois and abelian for each $i$, then $\bigwedge_i \HNP_K(L_i) \Longrightarrow \PHNP_K(L_1, L_2, \ldots, L_n)$.
    \label{abelianfinish}
\end{theorem}
\begin{proof}
For any abelian group $A$, we have $A^{\mathrm{der}} = 1$ and $A^{\mathrm{ab}} \cong A$. Since $G$ directly embeds into $\prod_i G/G^{(i)}$, which must be abelian, it follows that $G$ is abelian and thus $G^{\mathrm{der}} = 1$. Therefore by Proposition \ref{concretederived}, it suffices to show that there is no nontrivial $g \in G$ such that $g \in \bigcap_i G^{(i)}$. This follows from Lemma \ref{separatedk}.
\end{proof}

This result immediately gives us an analogue of the Hasse norm theorem for the projective Hasse norm principle:
\begin{cor}
    If $L_1, L_2, \ldots, L_n$ are cyclic Galois field extensions, then it follows that $\PHNP_K(L_1, L_2, \ldots, L_n)$ holds.
\end{cor}

Proposition \ref{concretederived} can also be used to tackle many nonabelian cases. To demonstrate this, we show below that this proposition is sufficient to address the case where $G$ is dihedral.

\begin{theorem}
    If $G$ is dihedral and $L_i/K$ is Galois for each $i$, then $\bigwedge_i \HNP_K(L_i) \Longrightarrow \PHNP_K(L_1, L_2, \ldots, L_n)$.
\end{theorem}
\begin{proof}
Let $G=D_m=\langle r,s:r^m=s^2=(sr)^2=1\rangle$. For $m\leq2$ the group is abelian, so assume $m\geq3$. We have $G^{\mathrm{der}}=\langle r^2\rangle$. By Proposition~\ref{concretederived}, it suffices to rule out $g\in\bigcap_iG^{\mathrm{der}}G^{(i)}$ with $g\notin G^{\mathrm{der}}$.

If $g$ is a rotation, then $m$ is even and its exponent is odd. Each $G^{(i)}$ then contains an odd power of $r$, hence contains the common involution $r^{m/2}$. This contradicts Lemma~\ref{separatedk}.

Otherwise each $G^{(i)}$ contains a reflection $r^ks$. Normality gives $r^{k+2}s= r(r^ks)r^{-1}\in G^{(i)}$, and hence $r^2\in G^{(i)}$. Since $m\geq3$, this again contradicts Lemma~\ref{separatedk}.
\end{proof}

By running the Oscar code in Appendix \ref{secondcode}, we found that in the groups we investigated, the conditions of Proposition \ref{concretederived} held in all groups $G$ whose order is fourth-power-free. Before giving the proof in \S\ref{sec:quarticfree}, we mention a few immediate results supporting this observation. 

\begin{theorem}
If each $L_i/K$ is Galois and $G$ has order $p^3$ for some prime $p$, then
\[\bigwedge_i\HNP_K(L_i)\Longrightarrow\PHNP_K(L_1,\ldots,L_n).\]
\label{orderp^3}
\end{theorem}
\begin{proof}
If $G$ is abelian, use Theorem~\ref{abelianfinish}. Otherwise $G^{\mathrm{der}}\ne1$. If some $G^{(i)}=1$, then $G^{\mathrm{der}}G^{(i)}=G^{\mathrm{der}}$ and Proposition~\ref{concretederived} applies directly. If all $G^{(i)}$ are nontrivial, every quotient $G/G^{(i)}$ has order at most $p^2$ and is abelian. Thus every $G^{(i)}$ contains $G^{\mathrm{der}}$, contradicting Lemma~\ref{separatedk}.
\end{proof}

\begin{prop}
    Suppose $G \cong \mathbb{Z}/p\mathbb{Z} \times H$ where $H$ is a nonabelian group of order $p^3$. Then there exists normal subgroups $G^{(1)},G^{(2)}$ with trivial intersection such that $G^{\mathrm{der}}G^{(1)} \cap G^{\mathrm{der}}G^{(2)} \neq G^{\mathrm{der}}$.
\end{prop}
\begin{proof}
    Any $p$-group has a nontrivial center (by the class equation), so $H$ has a central subgroup $H'$ of order $p$. Since $H/H'$ is abelian, it follows that $H^{\mathrm{der}} = H' \cong \mathbb{Z}/p\mathbb{Z}$. Suppose $G$ is isomorphic to $H \times \langle \alpha \rangle$ where $\alpha$ has order $p$, and let $\beta$ be a generator of $H^{\mathrm{der}}$. Now, choose $G^{(1)} = \langle \alpha \rangle, G^{(2)} = \langle \alpha \beta \rangle$. We have that $G^{(1)} \cap G^{(2)}$ is trivial and $G^{\mathrm{der}}G^{(1)} = G^{\mathrm{der}}G^{(2)} = \langle \alpha, \beta \rangle$. However, $G^{\mathrm{der}} = \langle \beta \rangle \subsetneq \langle \alpha, \beta \rangle$, as desired.
\end{proof}
This shows that the sufficient group criterion can fail at order $p^4$; by itself it does not prove failure of $\PHNP$. For $p=2$, Example~\ref{galoiscounterexample} gives an arithmetic realization with nontrivial projective norm obstruction.

\subsection{The Case of Galois Groups of Fourth-Power-Free Order}
\label{sec:quarticfree}

In fact, Proposition \ref{concretederived} can also be applied to generalize Theorem \ref{orderp^3} to any group whose order is fourth-power-free. Before we do so, let us introduce a few cohomological facts. 

\begin{prop}
    If $G$ is a finite group and $p$ is a prime dividing $|H^2(G, \mathbb{Q}/\mathbb{Z})|$, then $p^2$ divides $|G|$.
    \label{schur}
\end{prop}
\begin{proof}
Suppose for the sake of contradiction that $p^2 \nmid |G|$. Let $P$ be a Sylow $p$-subgroup of $G$, which must have order either $1$ or $p$. Now, the composition of the corestriction and restriction maps
\[H^2(G, \mathbb Q/\mathbb Z) \overset{\mathrm{res}}{\longrightarrow} H^2(P, \mathbb Q/ \mathbb Z) \overset{\mathrm{cores}}{\longrightarrow} H^2(G, \mathbb Q/ \mathbb Z)\]
is multiplication by $[G : P]$, which is not divisible by $p$. Hence, the restriction map is injective on the $p$-primary component of $H^2(G, \mathbb Q/\mathbb Z)$. However, $P$ is cyclic, so $H^2(P, \mathbb Q/ \mathbb Z) \simeq \frac{\mathbb Q/ \mathbb Z}{|P|(\mathbb Q / \mathbb Z)} = 0$, implying the $p$-primary part of $H^2(G, \mathbb Q/\mathbb Z)$ is also zero. However, this contradicts the existence of an element of order $p$ in $H^2(G, \mathbb{Q}/\mathbb{Z})$, so $p^2 \mid |G|$ as desired.
\end{proof}

Next, we prove a basic group order identity that will be used in the subsequent proofs.

\begin{lemma}
    Let $G$ be a finite group, and let $N$ and $M$ be normal subgroups of $G$ with trivial intersection. Set $R = NM$ and $Q = G/R$. Then $|G| = |N||M||Q|$.
    \label{productoforders}
\end{lemma}
\begin{proof}
    Note that $|R| = |NM| = \frac{|N||M|}{|N \cap M|} = |N||M|$, since $|N \cap M| = 1$. By Lagrange's theorem, $|Q||R| = |G|$, so $|Q||N||M| = |G|$.
\end{proof}

This lemma will be used in conjunction with Proposition \ref{schur} to show that when $N\cap M=1$ and $G^{\mathrm{der}}N \cap G^{\mathrm{der}}M \supsetneq G^{\mathrm{der}}$, then there must be some prime $p$ such that $p$ divides $|N|$ and $|M|$, and $p^2$ divides $|Q|$, so $|G|$ is not fourth-power-free. The following result gives a more general proof, where $N\cap M$ is not necessarily trivial. 

\begin{prop}
    If $|G|$ is not divisible by the fourth power of any prime, then for any normal subgroups $N, M \trianglelefteq G$, 
    \[G^{\mathrm{der}}N \cap G^{\mathrm{der}}M = G^{\mathrm{der}}(N \cap M).\]
    \label{twosubgroupcase}
\end{prop}
\begin{proof}
After quotienting by $N\cap M$, the derived subgroup becomes $G^{\mathrm{der}}(N\cap M)/(N\cap M)$ and the group order is still fourth-power-free. Thus it suffices to treat the case where $N$ and $M$ have trivial intersection. Suppose for the sake of contradiction that 
\[ G^{\mathrm{der}}N\cap G^{\mathrm{der}}M \neq G^{\mathrm{der}}. \]
Because $\frac{G^{\mathrm{der}}N\cap G^{\mathrm{der}}M}{G^{\mathrm{der}}}$ is a nontrivial subgroup of $G^{\mathrm{ab}}$, it contains an element $g$ of prime order, say, $p$. Then there exist $n \in N$ and $m \in M$ such that $g = nG^{\mathrm{der}} = mG^{\mathrm{der}}$. Since $n$ maps to an element of order $p$ in $G^{\mathrm{ab}}$, the order of $n$ (and hence the order of $N$) must be divisible by $p$. By identical reasoning, $p \mid |M|$. 

Since $N$ and $M$ intersect trivially, $NM \cong N \times M$. Reusing notation, let $R = NM$ and $Q = G/R$. Note that the canonical map $R \to G^{\mathrm{ab}}$ induces a homomorphism $\phi: R/[G, R] \to G^{\mathrm{ab}}$. Identifying $R/[G, R] \cong N/[G, N] \times M/[G,M]$, the element $(n[G, N], m^{-1}[G, M])$ has order divisible by $p$ and lies in $\ker \phi$, since its first component maps to $g$, which has order $p$. Hence $p\mid|\ker\phi|$.

Under the identifications
\[H^1(G, \mathbb Q/ \mathbb Z) \cong \mathrm{Hom}(G^{\mathrm{ab}}, \mathbb Q/ \mathbb Z) \text{ and } H^1(R, \mathbb Q/ \mathbb Z)^Q \cong \mathrm{Hom}(R/[G,R], \mathbb Q/ \mathbb Z),\] the five-term exact sequence for $1 \to R \to G \to Q \to 1$ yields
\[ \operatorname{Hom}(G^{\mathrm{ab}}, \mathbb Q/ \mathbb Z) \overset{\mathrm{res}}{\longrightarrow} \operatorname{Hom}(R/[G,R], \mathbb Q/ \mathbb Z) \overset{\delta}{\longrightarrow} H^2(Q, \mathbb Q/ \mathbb Z). \]

The restriction map is precomposition with $\phi$, so its image consists exactly of homomorphisms vanishing on $\ker(\phi)$: a character of $\operatorname{im}\phi$ extends to $G^{\mathrm{ab}}$ because $\mathbb Q/\mathbb Z$ is divisible. Because $\mathbb Q/ \mathbb Z$ is a divisible group, any homomorphism from $\ker(\phi)$ to $\mathbb Q/ \mathbb Z$ extends to all of $R/[G,R]$. Thus, 
\[ \frac{\operatorname{Hom}(R/[G,R], \mathbb Q/ \mathbb Z)}{\mathrm{Im}(\mathrm{res})} \cong \operatorname{Hom}(\ker(\phi), \mathbb Q/ \mathbb Z). \]
Since a finite abelian group and its character group have the same order, the order of this quotient is exactly $|\ker(\phi)|$, which is divisible by $p$.

By exactness, $\ker(\delta) = \mathrm{Im}(\mathrm{res})$. Applying the First Isomorphism Theorem implies that the quotient is isomorphic to $\mathrm{Im}(\delta)$, which is a subgroup of $H^2(Q, \mathbb Q/ \mathbb Z)$. Since $p \mid | \mathrm{Im}(\delta) |$, we have that $p \mid |H^2(Q, \mathbb Q/ \mathbb Z)|$. By Proposition \ref{schur}, $p^2 \mid |Q|$. Finally, Lemma \ref{productoforders} gives $|G| = |N||M||Q|$. Since $p \mid |N|, |M|$ and $p^2 \mid |Q|$, we conclude that $p^4 \mid |G|$.
\end{proof}

Note that as a byproduct of the proof of Proposition \ref{twosubgroupcase}, we have resolved the case where $N$ and $M$ have coprime order.

\begin{theorem}
Suppose that $G$ is finite and that $p^4\nmid |G|$ for every prime
$p$. If $G^{(1)},\dots,G^{(n)}\trianglelefteq G$ satisfy
\[\bigcap_{i=1}^n G^{(i)}=1,\]
then
\[\bigcap_{i=1}^n G^{\mathrm{der}}G^{(i)}=G^{\mathrm{der}}.\]
    \label{fourthpowerfree}
\end{theorem}

\begin{proof}
As with the two subgroups case, we prove the stronger assertion that for \textit{any} normal subgroups $G^{(1)}, \dots, G^{(n)}$, 
\[\bigcap_{i=1}^n G^{\mathrm{der}}G^{(i)}=G^{\mathrm{der}}\left(\bigcap_{i = 1}^n G^{(i)}\right).\] 
We do so by induction on $n$. The case $n = 1$ is clearly true, while the case $n = 2$ is equivalent to Proposition \ref{twosubgroupcase}. 
Suppose the result holds for $n - 1$ subgroups where $n \ge 3$, and set $H = \bigcap_{i = 1}^{n - 1} G^{(i)}$. Then,
\[\bigcap_{i = 1}^n G^{\mathrm{der}}G^{(i)} = \left( \bigcap_{i = 1}^{n - 1} G^{\mathrm{der}}G^{(i)} \right) \cap G^{\mathrm{der}} G^{(n)}.\]
By the inductive hypothesis, $\bigcap_{i = 1}^{n - 1} G^{\mathrm{der}}G^{(i)} = G^{\mathrm{der}}\left(\bigcap_{i = 1}^{n - 1} G^{(i)}\right) = G^{\mathrm{der}}H$, so the big intersection equals $G^{\mathrm{der}}H \cap G^{\mathrm{der}}G^{(n)}$. Since $H$ is the intersection of normal subgroups, it itself is a normal subgroup of $G$, so applying Proposition \ref{twosubgroupcase} yields
\[G^{\mathrm{der}}H \cap G^{\mathrm{der}}G^{(n)} = G^{\mathrm{der}}(H \cap G^{(n)}) = G^{\mathrm{der}}\left( \bigcap_{i = 1}^n G^{(i)}\right)\]
as desired.

\end{proof}

\begin{cor}
\label{fourthpowerfreefields}
If each $L_i/K$ is Galois and $G=\operatorname{Gal}(L_1\cdots L_n/K)$ has fourth-power-free order, then
\[\bigwedge_i\HNP_K(L_i)\Longrightarrow\PHNP_K(L_1,\ldots,L_n).\]
\end{cor}
\begin{proof}
Apply Lemma~\ref{separatedk}, Theorem~\ref{fourthpowerfree}, and Proposition~\ref{concretederived}.
\end{proof}
The order restriction is sharp: Example~\ref{galoiscounterexample} has $|G|=16$ and violates the implication. Since every positive integer less than $16$ is fourth-power-free, $16$ is the smallest possible degree of a compositum giving such a counterexample with Galois constituents.

\subsection{The Case of Coprime Degrees}
\label{coprimedegrees}

For normal subgroups $N,M\trianglelefteq G$ of coprime orders, the images of $N$ and $M$ in $G^{\mathrm{ab}}$ have trivial intersection. Thus $G^{\mathrm{der}}N\cap G^{\mathrm{der}}M=G^{\mathrm{der}}$.

\begin{prop}
If $N$ and $M$ are normal subgroups of $G$ with trivial intersection and coprime indices, then $G^{\mathrm{der}}N\cap G^{\mathrm{der}}M=G^{\mathrm{der}}$.
\label{coprimeindices}
\end{prop}
\begin{proof}
The index $[G:NM]$ divides both $[G:N]$ and $[G:M]$, so $G=NM$. Normality and $N\cap M=1$ give $G\cong N\times M$. Hence $G^{\mathrm{der}}=N^{\mathrm{der}}\times M^{\mathrm{der}}$, from which the assertion follows.
\end{proof}

\begin{prop}
\label{coprimedegreesresult}
Let $L_1,\ldots,L_n$ be finite separable extensions of a field $K$ whose degrees are pairwise coprime. Then the projective norm
\[K^\times\times\prod_i L_i^\times\longrightarrow(K^\times)^n\]
is surjective. In particular, if $K$ is global, $\PHNP_K(L_1,\ldots,L_n)$ holds without assuming the individual Hasse norm principles.
\end{prop}
\begin{proof}
Write $m_i=[L_i:K]$. By the Chinese remainder theorem choose integers $a_i$ with $a_i\equiv1\pmod{m_i}$ and $a_i\equiv0\pmod{m_j}$ for $j\ne i$. Given $(k_1,\ldots,k_n)\in(K^\times)^n$, set $s=\prod_i k_i^{-a_i}$. Each $sk_j$ is an $m_j$th power in $K^\times$, hence the norm of a scalar in $L_j^\times$. These scalars give a global common-scale solution.
\end{proof}

The Galois counterexample shows that the implication cannot hold without additional assumptions. Determining exactly which finite groups satisfy the intersection condition of Proposition~\ref{concretederived} remains a group-theoretic direction. For other groups, Propositions~\ref{numerator} and~\ref{xistructure} retain the arithmetic information in the decomposition groups and give a criterion for each specified family of extensions. The subgroup condition is sufficient, not necessary, for the projective Hasse norm principle.

\section{Acknowledgements}
The first two authors would like to thank MIT PRIMES for connecting them and providing the structure making such a project possible in the first place. Many thanks to Dr. Tanya Khovanova and Dr.~Felix Gotti for their suggestions and feedback.

\newpage

\appendix

\section{SageMath Code to Explore HNP and PHNP}
\label{code}
The following code uses the \texttt{GLattice} package of \cite{thomaspaper} and \texttt{MakeLattice} from \S\ref{sec:description}. With no decomposition groups supplied, the cohomology routine tests restriction to cyclic subgroups. A zero cyclic-restriction kernel implies that the arithmetic Tate--Shafarevich group is zero; a nonzero kernel alone does not prove failure of a norm principle. To obtain an arithmetic counterexample, one must also check every noncyclic decomposition group of the specified extension, as in \S\ref{sec:counterexamplehnptophnp}.

We retain all subgroup pairs, including pairs of isomorphic subgroups, since their embeddings in $G$ matter. No prime-power degree filter is used. The code returns the cyclic-restriction kernels for later inspection rather than declaring a counterexample merely because two groups differ. The package version and its standard induced-lattice basis convention should be recorded with any reported computation.
\begin{lstlisting}
def TestAll(G, normal_only=True):
    subg = (G.normal_subgroups() if normal_only
            else G.subgroups())
    for i, H1 in enumerate(subg):
        for H2 in subg[i:]:
            PHNP, HNP = MakeLattice(G, H1, H2)
            yield {
                'H1': H1, 'H2': H2,
                'cyclic_PHNP':
                    PHNP.Tate_Shafarevich_lattice(2),
                'cyclic_HNP':
                    HNP.Tate_Shafarevich_lattice(2)
            }

def HuntHNP(depth, normal_only=True):
    for degree in range(1, depth):
        for G in TransitiveGroups(degree):
            for result in TestAll(G, normal_only):
                yield G, result
\end{lstlisting}

\section{Oscar Code to Test for the Galois Case}
\label{secondcode}
These routines test the sufficient group condition for pairs of normal subgroups with trivial intersection. A reported pair is a failure of that group condition, not by itself a failure of the projective Hasse norm principle. The group order, library identifier, and both embedded subgroups should be retained with each result. The first routine works with intersections of subgroups and the second checks the equivalent condition in the quotient groups. A search through pairs with trivial intersection alone does not verify the general identity for all normal-subgroup intersections; that identity is proved in \S\ref{sec:quarticfree}. Record the Julia and OSCAR versions when reporting a search.

\begin{lstlisting}[language=Julia,keywordstyle=\bfseries\color{blue},stringstyle=\color{magenta},commentstyle=\color{ForestGreen},showstringspaces = false,basicstyle={\fontsize{9pt}{9pt}\ttfamily},aboveskip=0.3em,belowskip=0.1em]
function test(g,h1,h2)
      d = derived_subgroup(g)[1]
      h1d = sub(g,[gens(h1);gens(d)])[1] 
      h2d = sub(g,[gens(h2);gens(d)])[1]
      i = intersect([h1d,h2d])[1]
      if order(i)>order(d)
           true
      else
           false
      end
end


function testall(g)
      n = normal_subgroups(g)
      for h1 in n
           for h2 in n 
                if order(intersect([h1,h2])[1]) == 1
                     if test(g,h1,h2)
                          println("Group criterion fails: ", (g, h1, h2))
                     end
                 end
            end
       end
 end
      


function looptestall(n)
     numb = number_small_groups(n)
     for i in 1:numb
          println(i)
          testall(small_group(n,i))
          println(" ")
     end
end


function testv2(g,h1,h2)
      d = derived_subgroup(g)[1]
      q1, f1 = quo(g,h1)
      q2, f2 = quo(g,h2)
      d1 = derived_subgroup(q1)[1]
      d2 = derived_subgroup(q2)[1]
      for i in g
           if !(i in d)
                if (f1(i) in d1) && (f2(i) in d2)
                     return true
                end
            end
      end
      return false
end

function testallv2(g)
      n = normal_subgroups(g)
      l = []
      for h1 in n
           for h2 in n 
                if order(intersect([h1,h2])[1]) == 1
                     if testv2(g,h1,h2)
                          println("Group criterion fails: ", (g, h1, h2))
                          append!(l, [[g,h1,h2]])
                     end
                 end
            end
       end
      return l
 end
 
function looptestallv2(n)
     numb = number_small_groups(n)
     for i in 1:numb
          println(i)
          s = testallv2(small_group(n,i))
          for r in s
              println(r)
          end
          println(" ")
     end
end
\end{lstlisting}

%%%%%%%%%%%%%%
\bibliographystyle{amsalpha}
\bibliography{bibliography}
\bigskip

\end{document}